\documentclass[11pt]{amsart}

\usepackage{amsbsy, amsmath,amsfonts,amssymb,amsthm,amscd,amssymb,latexsym}
\usepackage{graphicx}
\usepackage[dvips]{epsfig}
\usepackage{graphics}
\usepackage{psfrag} 
\usepackage[usenames,dvipsnames]{xcolor}

\usepackage{cancel}

\usepackage{enumerate}
\usepackage{indentfirst}
\usepackage{euscript}

 \theoremstyle{plain}
 \newtheorem{theorem}{Theorem}
\newtheorem{proposition}{Proposition}
\newtheorem{lemma}{Lemma}

 \newtheorem{remark}{Remark}
\newtheorem{example}{Example}

 \theoremstyle{definition}

 \newtheorem{problem}{Problem}

\def \coes{\c c\~oes }
\def \ii{\'\i }
\def \cao{\c c\~ao }
\def \sao{s\~ao }
 
\def \ao{\~ao }

 \title[Centers of inscribed circles in  an elliptic billiard  ]{Centers of  inscribed circles in triangular orbits of an elliptic billiard}

 \author{  Ronaldo A. Garcia$^{*} $ }

  \thanks{* \; Partially supported by PRONEX/CNPq/FAPEG}

\begin{document}
 \maketitle
 \textwidth=125 mm  
  \textheight=195 mm

  \begin{abstract}  The locus of   centers of   inscribed  circles in triangles,  the 3-periodic orbits of an elliptic  billiard,   is also an ellipse.
 In this work we obtain the canonical equation of this ellipse,   complementing  the  previous results obtained by 
  O. Romaskevich in \cite{olga}.
 
\end{abstract}
  
\vskip .2cm

    \noindent{\bf MSC (2010). } 
  Primary 51M04  Secondary 	37D50;	53A15, 14H50,	14E05  
   \vskip .2cm
   
  \noindent {\bf Keywords.}  Elliptic billiard, triangular orbit, inscribed circle, incenter locus.

 \section{Introduction }
 
 Let 
 $\mathcal E$ be an ellipse and a point  $ p_1\in \mathcal{E}$.  Consider the triangle
  $\Delta(p_1)=\{p_1,p_2,p_3\} $  inscribed in  $\mathcal E$  
  that defines the 3-periodic of the billiard 
   (the normal vector of the ellipse at the $p_i$ is a bisectrix of $\Delta(p_1)$), associated to $\mathcal E$. See Fig. \ref{fig:bi1}. 
   
   Next consider the geometric locus
  $\mathcal{E}_c=\{C(p_1), p_1\in \mathcal E\} $, where  $C(p_1)$ 
  is defined as the center of the inscribed circle in the triangle
    $\Delta(p_1)$. 
    
    It is well known, by Poncelet's Theorem, that the sides of the  family of triangles 
  $\Delta(p_1)$ are tangent to a smaller ellipse $\mathcal{E}_1$, confocal to the ellipse $\mathcal E$. See \cite{taba} and Fig. \ref{fig:elip_poncelet}.
 
 The main goal of this work is to obtain, using only techniques of  real analytic and differential geometry, that 
  $\mathcal{E}_c$ is an ellipse. We  present its canonical algebraic equation and also a parametrization which is a  triple covering.  Other geometrical related facts are also considered.
 
  The fact that 
  $\mathcal{E}_c$ is an ellipse were established by  O. Romaskevic, see   \cite{olga}, using techniques of complex algebraic geometry.  See section 1 of the mentioned paper. Also we observe that in the work \cite{olga} no information about the shape of $\mathcal{E}_c$ was achieved.  
 
 Our strategy is to obtain the  parametrization of the geometric locus  $\mathcal{E}_c$ 
 and show that is a parametrized curve having  positive constant affine curvature (therefore is an ellipse).
 This is the essence of   Proposition  \ref{prop:kafim}.
 In  Theorem  \ref{th:cc} we describe explicitly the axes of  ellipse
  $\mathcal{E}_c$ and its canonical equation. Also it is worth to mention that the two ellipses $\mathcal{E}_1$ and $\mathcal{E}_c$ are similar (the axes are proportional).
  
  We make use of algebraic computational manipulators to perform the calculations.
  They are long in general, but  can be checked  by hand.

 \section{Preliminaries}
 
Consider  an ellipse  $\mathcal E$ given implicitly by  $h(x,y)=x^2/a^2+y^2/b^2-1=0$. To fix the arguments it will be supposed that  
 $a>b>0$  and so  the foci of $\mathcal E$ are given by $(\pm \sqrt{a^2-b^2} ,0)=(\pm c, 0).$

Given a point  $p_1=(x_1,y_1)\in \mathcal E$ we have that
$T_1=(-\frac{y_1}{b^2},\frac{x_1}{a^2})$ and  $N_1=(-\frac{x_1}{a^2},-\frac{y_1}{b^2})$ is a positive orthogonal frame (basis) of   $\mathbb R^2$.

\begin{proposition}\label{prop:3bilhar} For any point $p_1=(x_1,y_1)\in \mathcal E$  there exists a single
triangle   $\Delta(p_1)= \{p_1,p_2,p_3\}$ inscribed in   $\mathcal E$ such that $\Delta(p_1)$ is a  billiard (3-periodic orbit) inscribed in  $\mathcal E$. See Fig. \ref{fig:bi1}.

 \begin{figure}[htbp]\label{fig:bi1}
\begin{center}
 \def\svgwidth{0.95\textwidth}
    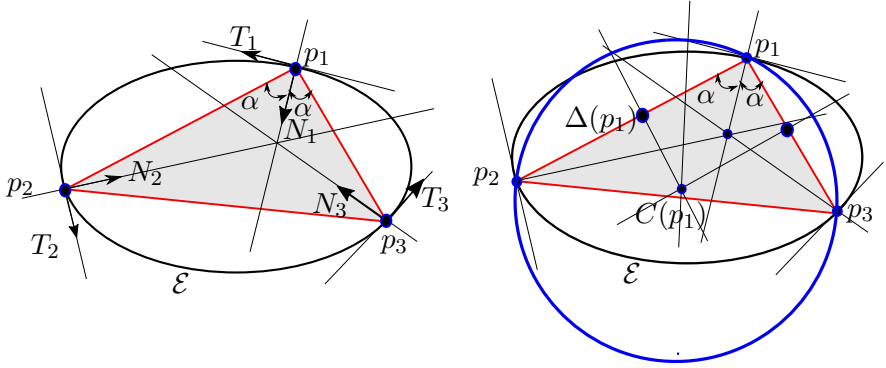
       \caption {Triangular billiard orbit inscribed in the ellipse   $\mathcal E$ (left) and center $C(p_1)$ of the inscribed circle in the triangle
           $\Delta(p_1)$ (right).}
   \end{center}
 \end{figure}
 
 \end{proposition}

\begin{proof} The proof  consists in the calculations of  angles of incidence and reflection  that  the triangle $\Delta=\{p_1,p_2,p_3\}$ (the 3-periodic orbit), positively oriented,  makes at its apexes with the normal vector of the ellipse.
See Fig. \ref{fig:bi1}.

Let $\alpha$ be the angle of incidence and reflection at the point $p_1$. So, we have defined two directions   $d_{12}=\sin\alpha T_1+\cos \alpha N_1$  and $d_{13}=-\sin\alpha T_1+\cos \alpha N_1$.

Evaluating the intersection of the straight lines 
 $p_1+td_{12}$ and $p_1+td_{13}$  with the ellipse  $\mathcal E$ we obtain the points  $p_2$ and  $p_3$, defining a triangle  $\Delta=\{p_1,p_2,p_3\}$ inscribed in the ellipse.

We have that  $p_2=(x_2, y_2)$, where  $x_2=\frac{p_{2x}}{q_2}$ and  $y_2=\frac{p_{2y}}{q_2}$ are given by:

 \begin{equation}\aligned 
p_{2x}=&-{b}^{4} \left(  \left(   a^2+{b}^{2}\right)\cos^{2}t -{a}^{2}  \right) x_1^{3}-2{a}^{6} \,\cos t \sin  t  y_1^{3}\\
+&{a}^{4} \left(  ({a
}^{2}-3\, {b}^{2}) \cos^{2} t +{b}^{2}
 \right) {x_1}\,y_1^{2}-2\,{a}^{4}{b}^{2} \cos t  \,\sin t  x_1^{2}{y_1}
\\
p_{2y}=& 2{b}^{6} \,\cos t\sin t\,   x_1^{3}-{
a}^{4}  \left(  \left(   a^2+{b}^{2}\right)\cos^{2}t -{b}^{2}  \right)  y_1^{3}\\
+& 2\,{a}^{2} {b}^{4}\cos t \sin
 t\; {x_1} y_1^{2} +{b}^{4}\left(  ({b
 }^{2}-3\, {a}^{2}) \cos^{2} t +{a}^{2}
  \right) x_1^{2}{y_1}
\\
q_2=&{b}^{4} \left( a^2-(a^2-b^2)\cos^2t  \right)
x_1^{2}+{a}^{4} \left(  {b}^{2}+(a^2-b^2)\cos^2t 
 \right) y_1^{2}\\
 -&2\, {a}^{2}{b}^{2} \left({a}^{2} -{b}^{2} \right)\cos t\sin t \; {x_1}\,{
y_1}
\endaligned
\end{equation}

Also, we have that $p_3=(x_3, y_3)$, where $x_3=\frac{p_{3x}}{q_3}$ and $y_3=\frac{p_{3y}}{q_3}$ are given by:

 \begin{equation}\aligned 
p_{3x}=& {b}^{4} \left( {a}^{2}- \left( {b}^{2}+{a}^{2} \right) \right)
 \cos^{2}t  x_1^{3}+2\,
 \cos t\, {a}^{6}\sin t\, y_1^{3}\\
 +&{a}^{4} \left( 
  \cos^{2}t \left( {a}^{2}-3\,{b}^{2}
 \right) +{b}^{2} \right) { x_1}\, y_1^{2}+2\,{a}^{4}{b}^{2} \cos  t\sin t\,   x_1^{2}{ y_1}
\\
p_{3y}=& -2\,\cos t {b}^{6}\sin t\, x_1^{3}+
{a}^{4} \left( {b}^{2}- \left( {b}^{2}+{a}^{2} \right)   \cos^{2} t\right)\,  y_1^{3}\\
-&2\,{a}^{2}  {b}^{4}\cos
 t \sin t\,  x_1 y_1^{2}+
{b}^{4} \left( {a}^{2}+ \left( {b}^{2}-3\,{a}^{2} \right)  \left( \cos
 t \right) ^{2} \right) {{ x_1}}^{2}{ y_1}
\\
q_3=& {b}^{4} \left( {a}^{2}- \left(a^2 -{b}^{2}  \right)   \cos^{2}t  \right) x_1^{2}+{a}^{4} \left( {b}^{2}+ \left( a^2-{b}^{2}  \right)  \cos^{2}t \right)  y_1^{2}\\
&+2\,{a}^{2}{b}^{
2} \left( {a}^{2}-{b}^{2} \right) \cos t \sin t\, { x_1}\,{ y_1}
\endaligned
\end{equation}

Therefore  the segments $p_{12}=p_2-p_1$ and   $p_{13}=p_3-p_1$ are, respectively, oriented by $d_{12}=\sin\alpha T_1+\cos \alpha N_1$ and $d_{13}=-\sin\alpha T_1+\cos \alpha N_1$.

Performing the calculations of  angles of incidence and reflection  at the points
  $p_2$ and $p_3$ and supposing that the triangle   $\Delta$ is a 3-periodic orbit of the billiard we obtain the two equations,

$$\frac{ \langle p_1-p_2, N_2\rangle }{ |p_1-p_2| | N_2| }=\frac{ \langle p_3-p_2, N_2\rangle }{ |p_3-p_2| | N_2| },\;\;
\frac{ \langle p_1-p_3, N_3\rangle }{ |p_1-p_3| | N_3| }=\frac{ \langle p_2-p_3, N_3\rangle }{ |p_2-p_3| | N_3| }.$$

A long and straightforward  calculation,  corroborated by algebraic symbolic computation, shows that the above equations have a common solution defined by:

{\small   
\begin{equation}\label{eq:coa}
\aligned 
%
       c^4& \left(  
      {b}^{4}x_1^2+ {a}^{4} y_1^2\right)^{2}   \cos^{4} \alpha+2\,{a}^{4}{b}^{4} \left( {a}^{2}+{b}^{2}
      \right)  \left(  {b}^{4}x_1^2+ {a}^{4} y_1^2\right) 
        \cos^{2}\alpha-3\,{a}^{8}{b}^{8}=0.\\
           c^4  &|T_1|^4  \cos^{4} \alpha+2\, \left( {a}^{2}+{b}^{2}
          \right)   |T_1|^2
            \cos^{2}\alpha-3\, =0, \;\; |T_1|^2=\frac{x_1^{ 2}}{a^4} +\frac{y_1^{ 2}}{b^4}.\\
        \endaligned 
\end{equation}}
 
 The equation
  \eqref{eq:coa} has a single solution $\alpha\in (0,\frac{\pi}2)$
  and therefore defines uniquely the triangle $\Delta(p_1)$.
This ends the proof.
\end{proof}

By Poncelet's Theorem, see for example \cite{drag}, \cite{gla}, \cite{grif}, 
 and \cite{taba}, the 3-periodic orbits given by the triangles $\Delta(p_1)$ are tangent to an ellipse  $\mathcal{E}_1$ which is confocal with the   ellipse $\mathcal{E}$.

Concretely we have the following result.

\begin{proposition}\label{prop:elipponcelet}
The smaller ellipse 
  $\mathcal{E}_1$
  of the 3-periodic billiard, confocal to the ellipse
  $\mathcal E$,   is given by:

\begin{equation}
\label{eq:eponce}
\aligned h_1(x,y)=&  \frac{x^2}{a_1^2}+\frac{y^2}{b_1^2}-1=0,\\
a_1=&-a (b^2-  \sqrt {{b}^{4}-{a}^{2}{b}^{2}+{a}^{4}})/(a^2-b^2)>0\\
%
 b_1=&b (a^2- \sqrt {{b}^{4}-{a}^{2}{b}^{2}+{a}^{4}})/(a^2-b^2)>0\\
%
%
\endaligned
\end{equation}

 
 \begin{figure}[htbp]\label{fig:elip_poncelet}
\begin{center}
  \def\svgwidth{0.75\textwidth}
     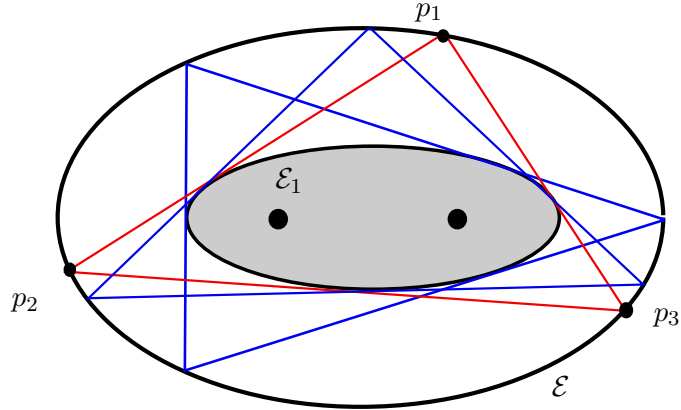
       \caption { Confocal ellipses $\mathcal E$ and  $\mathcal{E}_1$ and a triangular orbit
       $\{p_1,p_2,p_3\}$.}
   \end{center}
 \end{figure}
 
\end{proposition}

\begin{proof} It follows directly from Proposition   \ref{prop:3bilhar}, obtaining the triangles (3-periodic  orbits) generated by the points $(\pm a, 0)$ and  $(0,\pm b)$ and computing   the  smaller ellipse.
\end{proof}

\begin{lemma}\label{lem:cc}  The center of  the  inscribed circle in the triangle
  $\Delta=\{p_1,p_2,p_3\}$, $p_i=(x_i, y_i)$,   is given by $(x_c,y_c)$, where:

\begin{equation}
\aligned 
x_c=& \frac 12\,{\frac { (x_1^2+y_1^2)(y_2-y_3)+ (x_2^2+y_2^2)(y_3-y_1)+  (x_3^2+y_3^2)(y_1-y_2)}{ x_1(y_2-y_3)+x_2(y_3-y_1)+x_3(y_1-y_2)}},\\
y_c=&  \frac 12\,{\frac{(x_1^2+y_1^2)  (x_3 - x_2)+ (x_2^2+y_2^2)(x_1-x_3)+   (x_3^2+y_3^2) (x_2-x_1) }{ y_1(x_3-x_2)+y_2(x_1-x_3)+y_3(x_2-x_1)}}
\endaligned
\end{equation}

\end{lemma}
\begin{proof} Direct calculation. We obtain the intersection between   medial lines   $(p_1+p_2)/2+t (p_2-p_3)^\perp$ and  $(p_1+p_3)/2+t (p_1-p_3)^\perp$, bisectors of triangle $\Delta(p_1)$,  where $(p_i-p_j)^\perp $ is a vector orthogonal to $p_i-p_j$.
\end{proof}

\begin{proposition} \label{prop:ccpa} 
Let   $C(p_1)$ the center of  inscribed circle in the triangle   $\Delta(p_1)$
that defines the 3-periodic orbit of the billiard  associated to the confocal ellipses
 $\mathcal E$ and $\mathcal{E}_1$.  Then the locus   $\{ \ C(p_1), \;p\in \mathcal{E}\}$ is parametrized by  
 $ (x_c,y_c)$, 
 where:
 
 \begin{equation}
 \label{eq:parac}
\aligned x_c=& 
  \frac{x_1}{a}\left( \frac{A^3 x_1^2+ a_{02}
   y_1^2}{A^2x_1^2+B^2 y_1^2}\right), \;\;\;
 y_c= 
  \frac{y_1}{b}\left(\frac{ b_{20}
   x_1^2+B^3y_1^2}{A^2x_1^2+B^2 y_1^2}\right),\\
   A =& \,{\frac {{a}^{2}-\sqrt {{a}^{2}c^2+{b}^{4}}}{2a}},\;\; B =\,{\frac {{b}^{2}-\sqrt { {a}^{2}{c}^{2}+{b}^{4}}}{2b}},\\
  a_{02}(a,b)=&  {\frac {a \left(  \left( 3\,{a}^{3}+a{b}^{2} \right) A -{b}^{2}c^2 \right) }{4{b}^{2}}},\;\;\;
   b_{20}(a,b)=   \,{\frac {b \left(\left(   {a}^{2}b+3\,{b}^{3} \right) B+{a}^{2
     }{c}^{2} \right) }{4{a}^{2}}}.
  \endaligned \end{equation}

\end{proposition}

 \begin{proof} It follows directly from    Lemma \ref{lem:cc} and  Proposition  \ref{prop:3bilhar}.
 
 We observe that
 $$\aligned x_c=&  \frac{x_1F_1}{2a^2F_3} =\frac{x_1}{2a^2} \frac{a_{20}(a,b)x_1^2+a_{02}(a,b)y_1^2}{p(a,b)x_1^2+q(a,b)y_1^2}\\
 y_c=&  \frac{y_1F_2}{2b^2F_3} =\frac{y_1}{2b^2} \frac{ b_{20}(a,b)x_1^2+b_{02}(a,b)y_1^2}{p(a,b)x_1^2+q(a,b)y_1^2}\\
 a_{20}(a,b)=& b_{02}(b,a), \;\; a_{02}(a,b)= b_{20}(b,a).
\endaligned$$

\noindent where,

\begin{equation}\label{eq:param}
\aligned
 F_1=&   \left(  {b}^{2} \left( 4\,{a}^{4}-{a}^{2}{b}^{2}+{b}^{4} \right) { 
 A_2}-{a}^{2}{b}^{2} \left( 4\,{a}^{4}-3\,{a}^{2}{b}^{2}+3\,{b}^{4} \right) \right)
 x_1^{2}\\&+{a}^{4} \left(  \left( 3\,{a}^{2}+{b}^{2}
  \right) {  A_2}-3\,{a}^{4}+{a}^{2}{b}^{2}-2\,{b}^{4} \right) y_1^{2} 
 \\
F_2=&   \left( {b}^{4} \left( {a}^{2}+3\,{b}^{2} \right) {  A_2}-{b}^{4}
 \left(2\,{a}^{4} -{a}^{2}{b}^{2}+3\,{b}^{4} \right)  \right)x_1^{2}\\
 +& \left( {a}^{2} \left({a}^{4} -{a}^{2}{b}^{2}+4\,{b}^{4}
 \right) {  A_2}-{a}^{2}{b}^{2} \left( 3\,{a}^{4}-3\,{a}^{
2}{b}^{2}+4\,{b}^{4} \right)  \right) y_1^{2} 
 \\
F_3=&  \left( 2\,{a}^{2}{b}^{2}{ A_2}-{b}^{2} \left( 2\,{a}^{4}-{a}^{2}{b}
^{2}+{b}^{4} \right)  \right) x_1^{2}\\
+& \left( 2\,{a}^{2}{b}^{2}
{  A_2}-{a}^{2} \left( 2\,{b}^{4}+{a}^{4}-{a}^{2}{b}^{2} \right) 
 \right) y_1^{2}\\
 A_2=& \sqrt{a^2(a^2-b^2)+b^4}=\sqrt{a^2c^2+b^4}
\endaligned\end{equation}
Simplifying the above equation it follows the result.
 \end{proof}

In terms of the axes of ellipses 
 $\mathcal E$ and  $\mathcal{E}_1$ we have the following proposition.
 
\begin{proposition}  Let $p_1=(x_1,y_1)\in\mathcal{E}$ and   $C(p_1)$ the center of the  inscribed circle in the triangle   $\Delta(p_1)$
that defines the 3-periodic orbit of the billiard  associated to the confocal ellipses
 $\mathcal E$ and $\mathcal{E}_1$.  Then the locus   $\{ C(p_1):\; p_1\in \mathcal{E}\}$ is parametrized by  $ (x_c,y_c)$, where:
 
\begin{equation}
\aligned 
x_c=&\frac{x_1}{4} \frac{{b}^{2} \left( b_1^{2}({b}^{4}-{b}^{2} {a}
^{2}+4\, {a}^{4})-{b}^{6} \right) x_1^{2}+{a}^{4} \left( {b
}^{2}b_1^{2}-{b}^{4}+3\,b_1^{2}{a}^{2} \right) y_1^2}{(b_1^2x_1^2+a_1^2y_1^2) a^4b^2}\\
y_c=& \frac{y_1}{4} \frac{{b}^{4} \left(  a_1^{2}{a}^2-{a}^{4} + 3\,a_1^{2}{b}^{2}
\right) x_1^{2}+ {a}^{2} \left( a_1^{2}({a}^{4}-  {b}^{2}{a}^{2}+4\,{b}^{4})-a^6 \right) y_1^{2} }{(b_1^2x_1^2+a_1^2y_1^2) a^2b^4}
\endaligned
\end{equation}

\end{proposition}

\begin{proof}
It follows from Propositions    \ref{prop:elipponcelet} and \ref{prop:ccpa}.
\end{proof}

\section{Affine curvature of plane curves }

The affine group of transformations of the plane is generated by translations and linear isomorphisms.
 
 As in the Euclidean case, it is natural to consider length and curvature of plane curves which are invariant by the special affine group.
 
Consider a regular plane curve
  $c(s)=(x(s),y(s))$ in the affine plane  $\mathbb A^2$  with the canonical area form  $\omega=dx\wedge dy$ and suppose that $[c^\prime(s),c^{\prime\prime}(s)]=1$. This is always possible when
 $c$ is a convex curve, i.e., its Euclidean curvature is positive.

Differentiating this equation it follows that
  $ [c^\prime(s),c^{\prime\prime\prime}(s)]=0$  
  and, as  $c^\prime \ne 0$, it is obtained.

$$c^{\prime\prime\prime}(s)+k_a(s)c^{\prime }(s)=0  \;\;\;\Rightarrow\;\;\; k_a=[c^{\prime\prime }(s),c^{\prime\prime\prime}(s)].$$
Above $[u,v]=u_1v_2-u_2v_1 $ 
is the determinant of   matrix formed by the vectors
   $u=(u_1,u_2)$ and  $v=(v_1,v_2)$.

The function  $k_a$  is called  the {\it   affine curvature } of  $c$. See  \cite{no} and  \cite{spivak}.
 
With respect a parametrization 
  $\gamma(u)=c(s(u))$ it follows that

$$\aligned \gamma^\prime =& c_s s^\prime \\
 \gamma^{\prime\prime} =& c_{ss} {s^\prime}^2+ c_{s} {s^{\prime\prime}}\\
  \gamma^{\prime\prime\prime} =& c_{sss} {s^\prime}^3+  3 c_{ss} {s^\prime}  {s^{\prime\prime}} + c_s s^{\prime\prime\prime}\endaligned$$

\indent where,  $c_s=dc/ds=c^\prime$, $c_{ss}=d^2c/ds^2 =c^{\prime\prime} $  and $c_{sss}=d^3c/ds^3=c^{\prime\prime\prime}.$

 Moreover,

$$\aligned {s^\prime}^3=&[\gamma^\prime,\gamma^{\prime\prime}],\\
3{s^\prime}^2s^{\prime\prime}=&[\gamma^\prime,\gamma^{\prime\prime\prime}]\\
6{s^\prime} {s^{\prime\prime}}^2+3{s^\prime}^2 s^{\prime\prime\prime}=&[\gamma^\prime,\gamma^{\prime\prime\prime\prime}]+[\gamma^{\prime\prime},\gamma^{\prime\prime\prime}].
\endaligned
$$

Now, using that 
 $k_a (s)=[c_{ss}, c_{sss}]$,  we obtain, in terms of the parametrization     $\gamma(u)$, that :

\begin{equation}
\label{eq:kafim}
k_a(u)=\frac{4[\gamma^{\prime\prime},\gamma^{\prime\prime\prime}]+[\gamma^\prime,\gamma^{\prime\prime\prime\prime}]}
{3[\gamma^\prime,\gamma^{\prime\prime }]^{\frac 53}}-\frac 59 \frac{ [\gamma^\prime,\gamma^{\prime\prime\prime}]^2}{[\gamma^\prime,\gamma^{\prime\prime }]^{\frac 83}}.
\end{equation}

%
%
%
%

\begin{remark}
The affine curvature of the ellipse $\mathcal E$ is given by ${1}/{(ab)^{\frac{2}{3}}}.$
\end{remark}

\section{Main Results  }

\begin{proposition}\label{prop:kafim}
The affine curvature 
  $k_a$ 
  of the curve of centers $\{C(p_1):\; p_1\in \mathcal{E} \} $ of the inscribed circles in triangles of an elliptic billiard is constant and is given by:
{\small 
\begin{equation}
\aligned 
k_a^3=& \frac{1}{A^2B^2},\;
A=\,{\frac {{a}^{2}-\sqrt {{a}^{4}-{a}^{2}{b}^{2}+{b}^{4}}}{2a}},\;\; B=\,{\frac {{b}^{2}-\sqrt {{a}^{4}-{a}^{2}{b}^{2}+{b}^{4}}}{2b}}.
\endaligned
\end{equation}
}

\end{proposition}

\begin{proof} A long calculation,  confirmed by algebraic symbolic computation, gives  that the affine curvature of $C(p_1)$ is constant and positive.
 
In fact, using the parametrization given by equation
  \eqref{eq:param}, making $x_1=au$ and  $y_1=b\sqrt{1-u^2}$  and with the expression for the   affine curvature given by equation
  \eqref{eq:kafim}, it follows that:

$$\aligned k_a^3=&-\frac{16a^2b^2U^3}{V^8}\\
U=&-36493\,{a}^{6}{b}^{18}-6561\,{a}^{22}{b}^{2}-38546\,{a}^{20}{b}^{4}-
36493\,{a}^{18}{b}^{6}-62918\,{a}^{16}{b}^{8}\\-&67282\,{a}^{14}{b}^{10}\ 
- 
74444\,{a}^{12}{b}^{12}-67282\,{a}^{10}{b}^{14}\\&-62918\,{a}^{8}{b}^{16}
-38546\,{a}^{4}{b}^{20}-6561\,{a}^{2}{b}^{22}-13122\,{a}^{24}-13122\,{
b}^{24}\\
+&2\,{  A_2}\, \left( {b}^{2}+{a}^{2} \right)  \left( 81\,{b}^{
8}-20\,{b}^{6}{a}^{2}+134\,{b}^{4}{a}^{4}-20\,{b}^{2}{a}^{6}+81\,{a}^{
8} \right)\\
&  \left( 81\,{b}^{12}+20\,{b}^{10}{a}^{2}+119\,{b}^{8}{a}^{4
}+72\,{b}^{6}{a}^{6}+119\,{b}^{4}{a}^{8}+20\,{b}^{2}{a}^{10}+81\,{a}^{
12} \right)\\
V=& \left( 27\,{a}^{8}+20\,{b}^{2}{a}^{6}+34\,{b}^{4}{a}^{4}+20\,{b}^{6}{
a}^{2}+27\,{b}^{8} \right) {  A_2}\\
-& \left( {b}^{2}+{a}^{2} \right) 
 \left( 27\,{b}^{8}-20\,{b}^{6}{a}^{2}+50\,{b}^{4}{a}^{4}-20\,{b}^{2}{
a}^{6}+27\,{a}^{8} \right)
\endaligned $$

The above equation admits the factorization
 $k_a^3=1/(AB)^2$. This was confirmed by algebraic symbolic computation.

As the affine curvature $k_a$  is constant and positive it follows that the geometric locus 
 $\mathcal{E}_c=\{  C(p_1), p_1\in\mathcal{E}\}$
is an ellipse. 
\end{proof}

\begin{theorem} \label{th:cc} The locus  $\mathcal{E}_c= \{  C(p_1), p_1\in\mathcal{E}\}$ is and ellipse with canonical equation    $x^2/A^2+y^2/B^2=1$, where

\begin{equation}\label{eq:ce}
A = \,{\frac {{a}^{2}-\sqrt {{a}^{2}c^2+{b}^{4}}}{2a}}, \;\; \; B =\,{\frac {{b}^{2}-\sqrt { {a}^{2}{c}^{2}+{b}^{4}}}{2b}}.
\end{equation}
The foci of $\mathcal{E}_c $    are   $(0,\pm \frac{c^3}{2ab}).$
In relation to the axes  $a_1$ and  $b_1$ of the smaller ellipse $\mathcal{E}_1$  
it follows that
$4a^2b^2A^2= c^4 b_1^2 $ and  $  4a^2b^2B^2=c^4  a_1^2.$ In particular,  $\frac{b_1}{a_1}=\frac{A}{|B|}.$

If  $a>b\geq \frac a4\,\sqrt {-1+\sqrt {33}}\   $ the ellipse  $x^2/A^2+y^2/B^2=1$
is contained in the region bounded  by   ellipse
$x^2/a^2+y^2/b^2=1$.
\end{theorem}

\begin{proof} By Proposition \ref{prop:kafim} we have that the $\mathcal{E}_c$ is an ellipse.
 
This ellipse contains the five  points 
  $P_i, (i=1,\ldots 5)$, where
$P_1=\gamma(a,0),\; P_2=\gamma(-a,0),\; P_3=\gamma(0,b),\;P_4=\gamma(0,-b),\; P_5=\gamma(a/2,(\sqrt{3}/2 ) b).$

Elementary calculation, solving a linear system of equation, although long, shows that the ellipse is as affirmed.
 
With the hypothesis that  $a>b$, we have  $B^2>A^2$ and $C^2=B^2-A^2=\frac{(a^2-b^2)^3}{4a^2b^2}.$ Therefore the foci are given by   $(0,\pm C)$, where $C=c^3/(2ab).$

Another way to confirm the result is that, defining 
  $H(x,y)=x^2/A^2+y^2/B^2-1$, we obtain directly that 
$H(x_c(a\cos t, b\sin t), y_c(a\cos t, b\sin t))=0$.

The algebraic relations   $4a^2b^2A^2= b_1^2c^4$ and  $4a^2b^2B^2= a_1^2c^4$ 
follows directly from the definition of  quantities involved.

The  inequality  $B\leq b$ is  equivalent  to the   condition $ b\geq \frac a4\,\sqrt {-1+\sqrt {33}}.$ This ends the proof.
\end{proof}

\begin{proposition}\label{prop:3voltas}
The  curve  $\gamma(t)=(x_c(t), y_c(t))=$ $  (x_c (x_1,y_1),y_c(x_1,y_1)$, where  $\Gamma(t)=(x_1,y_1)=(a\cos t, b\sin t)$, $0\leq t\leq 2\pi$, is a triple covering of the ellipse    $\mathcal{E}_c$, i.e., defining $C:\mathcal{E}\to\mathcal{E}_c$ by $C(\Gamma(t))=\gamma(t)$ it follows that $C^{-1}(p_c)$ always  consists of three points for all $p_c\in\mathcal{E}_c.$

\end{proposition}
 \begin{figure}[htbp]
\begin{center}
  \def\svgwidth{0.50\textwidth}
    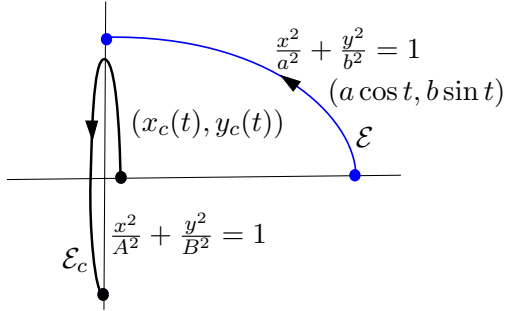
       \caption { Sketch of the ellipses  $\mathcal{E}$ (curve $\Gamma)$) and  $\mathcal{E}_c$ (curve $\gamma$ ) in the  interval  $[0,\pi/2]$.\label{fig:bilhar2} }
   \end{center}
 \end{figure}

\begin{proof} Consider a rational function  $F:\mathbb R^2\to \mathbb R^2$ of the type
$$F(x,y)=\left( x\,\frac{ax^2+by^2}{px^2+qy^2}, y\, \frac{cx^2+dy^2}{px^2+qy^2}\right)$$

We have that,
$$\text{det}({Jac(F)})=\frac{ac\,{x}^{4}+ \left( 3\,ad-bc \right) {x}^{2}{y}^{2}+bd\,{y}^{4}}{(px^2+qy^2)^2}.
$$
Applying the above result to the parametrization of ellipse $\mathcal{E}_c$ given by equation 
 \eqref{eq:parac} it follows that:
$$\aligned det(Jac(F))=&\frac{P}{Q^2}\\
P=&{-(b^4x_1^2+a^4y_1^2)(\alpha x_1^2+\beta y_1^2)}\\
\alpha=& {b}^{2} \left( 5\,{a}^{2}{b}^{6}+3\,{b}^{2}{a}^{6}+9\,{a}^{4}{b}^{4}+
12\,{a}^{8}+3\,{b}^{8} \right) \sqrt{a^4-a^2b^2+b^4} \\
-& {b}^{2} \left( 12\,{a}^{10}+4\,
{a}^{2}{b}^{8}+3\,{b}^{10}-3\,{a}^{8}{b}^{2}+12\,{b}^{4}{a}^{6}+4\,{a}
^{4}{b}^{6} \right) 
 \\
\beta=& {a}^{2} \left( 12\,{b}^{8}+5\,{b}^{2}{a}^{6}+3\,{a}^{8}+9\,{a}^{4}{b}^
{4}+3\,{a}^{2}{b}^{6} \right) \sqrt{a^4-a^2b^2+b^4}\\
-& {a}^{2} \left( 3\,{a}^{10}+4\,{a
}^{8}{b}^{2}-3\,{a}^{2}{b}^{8}+4\,{b}^{4}{a}^{6}+12\,{a}^{4}{b}^{6}+12
\,{b}^{10} \right)
 \\
Q=& 2a{b}^{3} \left( -{b}^{4}+{a}^{2}{b}^{2}-2\,{a}^{4}+2\,{a}^{2}\sqrt {{a}
^{4}-{a}^{2}{b}^{2}+{b}^{4}} \right) x_1^{2}\\
 +&2{a}^{3}b \left( -{a
}^{4}+{a}^{2}{b}^{2}-2\,{b}^{4}+2\,{b}^{2}\sqrt {{a}^{4}-{a}^{2}{b}^{2
}+{b}^{4}} \right) y_1^2.
\endaligned
$$
By algebraic manipulation we obtain that 
  $\alpha<0$ and  $\beta<0$ and therefore $det(Jac(F))>0$ for all $(x_1,y_1)\ne (0,0)$.

Moreover, we have that the function 
  $ x_c(x_1,y_1)=x_c(t)$  has six zeros, which are given  by  $x_1=0$ and  $y_1=\pm k_1(a,b)x_1$.
  The same for the function
 $y_c(x_1,y_1)=y_c(t)$ that vanishes at $y_1=0$ and  $y_1=\pm k_2(a,b) x_1$. Therefore, each coordinate function has 6 zeros in interval  $[0,2\pi)$.
 
 A typical example of this behavior is shown in  
  Fig. \ref{fig:xe} below. 

 \begin{figure}[htbp]
\begin{center}
  \def\svgwidth{0.60\textwidth}
     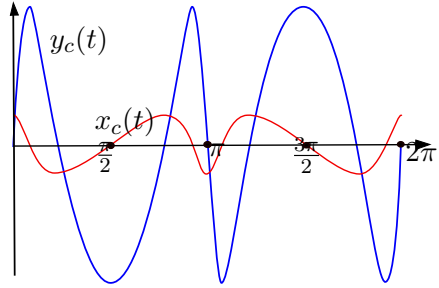
       \caption { Graphic of coordinate functions  $x_c$ and  $y_c$ of the  ellipse $\mathcal{E}_c$.\label{fig:xe}}
   \end{center}
 \end{figure}
Then $\gamma(t)=(x_c(a\cos t, b\sin t),y_c(a\cos t, b\sin t)$ is a triple covering of the   ellipse $\mathcal{E}_c$.
In fact, $C^{-1}(C(p_1))=\{p_1,p_2,p_3\}=\Delta(p_1).$

  The Fig. \ref{fig:bilhar2} shows the trace of   $\Gamma(t)=(a\cos t, b\sin t)$ and  $\gamma(t)$
  in the interval  $[0,\pi/2]$. This ends the proof.
\end{proof}

\begin{proposition}
The function 
$$F(a,b)=\left(\frac{ -a \left(b^2-  \sqrt {{b}^{4}-{a}^{2}{b}^{2}+{a}^{4}}\right)}{a^2-b^2},\frac{b \left(a^2- \sqrt {{b}^{4}-{a}^{2}{b}^{2}+{a}^{4}}\right)}{a^2-b^2}\right) $$
is a global diffeomorphism of the open region
 $R=\{(a,b) :\;  a>0, b> 0\; \text{ and}\;  a> b\}.$
%
%
\end{proposition}

\begin{proof} Direct calculation shows  that  $det(Jac(F))>0$ in the  region  $R$  and that $F(R)\subset R$. Therefore, by Inverse Function Theorem, $F$ is a local diffeomorphism.   
To obtain the result we observe that the equation   $F(a,b)=(x,y)$ is equivalent  to the system   
 $x^2-y^2=a^2-b^2, \;\; bx+ay=ab.$ With the hypothesis  $x> y$, the two hyperbolas in the plane $ab $ have only one  transversal point of intersection in the region    $R.$
In fact, writing  $c^2=x^2-y^2>0$ we have that $b=ay/(a-x)$ and  $a$ satisfies a polynomial equation    $p(a)={a}^{4}-2\,{a}^{3}x+2\,{c}^{2}ax-{c}^{2}{x}^{2}=0$ which has only  one positive solution.
 
  This follows from the fact the discriminant of  equation above is $-432\,{x}^{4}{c}^{4} \left( -{x}^{2}+{c}^{2} \right) ^{2}<0$ (see \cite{BP}) and therefore the quartic polynomial equation  $p(a)=0$ has only two real roots and in our case the positive one is bigger than   $x$. 
\end{proof}

\section{Conclusion}

In this work, inspired by the paper of O. Romaskevich \cite{olga},  we established in Theorem \ref{th:cc}  that the centers of circles  inscribed in triangular orbits of an elliptic billiard is also an ellipse $\mathcal{E}_c$. 
  
  The main contribution of this work  it was the  calculation of   the canonical equation of  ellipse  $\mathcal{E}_c$.  
   Also it is useful to observe that the axes of ellipse   $\mathcal{E}_c$ are proportional to the axes of a smaller ellipse $\mathcal{E}_1$ (confocal to the ellipse  $\mathcal{E}$)  associated to  the 3-periodic  billiard. 
  Also it was established in Proposition \ref{prop:3voltas} that the  parametrized ellipse   $\mathcal{E}$ (curve $\Gamma$) is a triple covering of   $\mathcal{E}_c$ (curve $\gamma$).

Finally we observe that the geometric  locus of barycenters and orthocenters of  the triangles $\Delta(p_1)$ (the 3-periodic  orbits  associated to an elliptic billiard)  are also ellipses. See   \cite{ga}.

 \vskip .5cm
 {\author  
 \noindent  Ronaldo A. Garcia\\
 Universidade  Federal de Goi\'as 
\\Instituto de Matem\'atica e Estat\'\i stica  
 \\ Campus Samambaia
\\ Goi\^ania - Goi\'as - Brasil
\\ CEP 74690-900  
\\ ragarcia@ufg.br }
 
%

 \end{document}